\documentclass{article}
\usepackage{typearea}
\typearea{14}
\usepackage{amsmath,amssymb}
\usepackage{bm}
\usepackage[pdftex]{graphicx}
\usepackage{float}
\usepackage[hang,small,bf]{caption}
\usepackage[subrefformat=parens]{subcaption}
\usepackage{amsthm}
\newtheorem{defi}{Definition}[section]
\newtheorem{theo}[defi]{Theorem}
\newtheorem{lemm}[defi]{Lemma}

\newtheorem{prop}[defi]{Proposition}
\newtheorem{co}[defi]{Corollary}

\newtheorem{rem}[defi]{Remark}

\usepackage{usebib}
\bibinput{kajihara_reference}
\usepackage{makeidx}

\usepackage{ascmac} 
\usepackage{fancybox}
\usepackage{float}
\usepackage{tabularx}
\usepackage{mathtools}
\usepackage{here}

\newcommand{\cals}{\mathcal{S}}

\newcommand{\R}{\mathbb{R}}
\newcommand{\Z}{\mathbb{Z}}

\newcommand{\N}{\mathbb{N}}
\newcommand{\Q}{\mathbb{Q}}

\newcommand{\per}{\mathrm{per}}

\newcommand{\calm}{\mathcal{M}}

\newcommand{\calc}{\mathcal{C}}
\newcommand{\xkr}{X_{k,\rho}}

\newcommand{\tilj}{\tilde{J}}

\title{Variational Structures for Infinite Transition Orbits \\of Monotone Twist Maps}
\author{Yuika Kajihara}
\date{\today}

\begin{document}
\maketitle

\begin{abstract}
In this paper, we consider chaotic dynamics and variational structures of area-preserving maps.
There is a lot of study on the dynamics of their maps and 
the works of Poincare and Birkhoff are well-known.
To consider variational structures of area-preserving maps,  we define a special class of area-preserving maps called {\it monotone twist maps}.
Variational structures determined from twist maps can be used for constructing characteristic trajectories of twist maps.
Our goal is to
prove the existence of an {\it infinite transition orbit},
which represents an oscillating orbit between two fixed points infinite times,
through minimizing methods.
\end{abstract}

\section{Introduction}
\label{section:intro}
%本結果の背景
%Rabinowitzの結果の説明
In this paper, we consider  chaotic dynamics and variational structures of area-preserving
%(we write AP in short)
maps.
The dynamics of such maps have been widely studied, with key findings by  Poincar\'e and Birkhoff.
There are a lot of related works, see \cite{Birkhoff1922, Birkhoff1966, Mather1986} for example.
To explore these variational structures, 
we define a special class of  area-preserving maps called {\it monotone twist maps}:

%def_monotonetwistmaps
\begin{defi}[monotone twist maps]
Set a map $f \colon \R / \Z \times [a,b] \to \R / \Z \times [a,b]$ and assume that $f \in C^1$ and
a lift $\tilde{f}$ of $f$ $\colon \R \times [a,b] \to  \R \times [a,b]$, $(x,y) \mapsto (f_1(x,y),f_2(x,y))(=(X,Y))$ satisfy the followings:
\begin{itemize}
  \item[($f_1$)] $\tilde{f}$ is  area-preserving, i.e., $dx \wedge dy = dX \wedge dY$;
  \item[($f_2$)] $\partial X / \partial y>0$ (twist condition), and
  \item[($f_3$)] Both two straight lines $y=a$ and $y=b$ are invariant curves of $f$, i.e. $f(x,a) - a =0$ and  $f(x,b) - b =0$ for all $x \in \R / \Z$.
\end{itemize}
Then $f$ is said to be a monotone twist map.
\end{defi}
By Poincar\'e's lemma, we get a generating function $h$ for a monotone twist map $f$
and  it satisfies
$
dh=YdX -ydx.
$
That is,
\[
y= \partial_1 h(x,X), \
Y= -\partial_2 h(x,X),
\]
where $\partial_1 = \partial/\partial x$ and $\partial_2 = \partial/\partial X$.
For the above $h$, by abuse of notation, we define $h \colon \R^{n+1} \to \R$ by:
\begin{align}
\label{action_n}
h(x_0,x_1, \cdots, x_n)=\sum_{i=1}^{n} h(x_i,x_{i+1}).
\end{align}
We can regard $h$ as a variational structure associated with $f$, because any critical point of $\eqref{action_n}$, say $(x_0,\cdots,x_n)$, gives us an orbit of $\tilde{f}$ by $y_i=-\partial h_1(x_i,x_{i+1})=\partial h_2(x_{i-1},x_{i})$.
Its relation implies that the orbit $\{ f^i(x_i,y_i)\}_{i \in \Z}$ corresponds to a {\textit{stationary configuration}} defined below.
This is known as the Aubry-Mather theory, which is so called because
Aubry studied critical points of the action $h$ in \cite{AubrDaer1983} and Mather developed the idea (e.g. \cite{Mather1982, Mather1987}).
%Bangert \cite{Bangert1988} investigates good conditions of $h$ for study in  {\it{minimal sets}}.
%Here we will give a summary of  \cite{Bangert1988}.
We briefly summarize Bangert's investigation of good conditions of $h$ for study in minimal sets \cite{Bangert1988}.

We consider the space of bi-infinite sequences of real numbers,
and define convergence of a sequence $x^n = (x^n_i)_{i \in \Z}\in \R^\Z$  to $x=(x_i)_{i \in \Z} \in \R^\Z$ by:
\begin{align}
\label{t_metric}
\lim_{n \to \infty} |x^n_i - x_i| = 0 \ ({}^{\forall} i \in \Z).
\end{align}
Now we treat a function $h$ satisfying {\it variational principle}.
All of the results in \cite{Bangert1988} we introduce assumes variational principle.
\begin{defi}[variational principle]
Let $h$ be a continuous map from $\R^2$ to $\R$
We call a function $h$ variational principle if it satisfies the following:
\begin{itemize}
  \item[($h_1$)] For all $(\xi,\eta) \in \R^2$, $h(\xi,\eta)=h(\xi+1,\eta+1)$;
  \item[($h_2$)] $\displaystyle{\lim_{\eta \to \infty} h(\xi,\xi+\eta) \to \infty \ (\text{uniformly in} \ \xi)}$;
  \item[($h_3$)] If $\underline{\xi}< \bar{\xi}$ and $\underline{\eta}< \bar{\eta}$, then
 $ h(\underline{\xi},\underline{\eta}) + h(\bar{\xi},\bar{\eta}) <   h(\underline{\xi},\bar{\eta}) + h(\bar{\xi},\underline{\eta}) $; and
  \item[($h_4$)] If $(x,x_0,x_1)$ and $(\xi,x_0,\xi_1)$ are minimal and $(x,x_0,x_1) \neq (\xi,x_0,\xi_1)$, then
  $(x-\xi)(x_1-\xi_1)<0$.
\end{itemize}
\end{defi}

In this paper,
we call an element $x=(x_i)_{i \in \Z}   \in \R^\Z$ a configuraition. 
There are distinctive configurations referred to as {\it{minimal configurations} }and {\it{stationary configurations}}.
\begin{defi}[minimal configuration/stationary configuration]
Fix $n$ and $m$ with $n<m$ arbitrarily.
A finite sequence $x=(x_i)_{i=n}^{m}$ is said to be minimal if,
for any (finite) configuration $(y_i)_{i=n_0}^{n_1} \in \R^{n_1-n_0+1}$ with $y_{n_0}=x_{n_0}$ and $y_{n_1}=x_{n_1}$,
\[
h(x_{n_0},x_{n_0+1}, \cdots, x_{n_1-1}, x_{n_1}) \le h(y_{n_0},y_{n_0+1}, \cdots, y_{n_1-1}, y_{n_1}),
\]
where  $n \le n_0 < n_1 \le m$.
A configuration $x=(x_i)_{i \in \Z}$ is called minimal if, for any $n<m$, we have $x=(x_i)_{i=n}^{m}$ is minimal.
Moreover, if $h \in C^1$, a configuration $x$ is called locally minimal or a stationary configuration if it satisfies:
\begin{align}
\label{stationary}
\partial_2 h (x_{i-1},x_i) + \partial_1h(x_i,x_{i+1})=0 \ ( {}^{\forall} i \in \Z).
\end{align}
\end{defi}

For $x=(x_i)_{i \in \Z} \in \R^\Z$, we define $\alpha^+(x)$ and $\alpha^-(x)$ by:
\[
\alpha^+(x) =\lim_{i \to \infty} \frac{x_i}{i},\
\alpha^-(x) =\lim_{i \to -\infty} \frac{x_i}{i}.
\]
We only discuss the case of $\alpha^+(x)=\alpha^{-}(x)$ in this paper.
\begin{defi}[rotation number]
If both $\alpha^+(x)$ and $\alpha^- (x)$ exist and $\alpha^+(x)=\alpha^-(x)(=:\alpha(x))$, 
then we call $\alpha(x)$ a rotation number of $x$.
\end{defi}

Let $\calm_\alpha$ be a minimal set consisting of minimal configurations with rotation number $\alpha$.
It is known that for any $\alpha \in \R$, the set $\calm_{\alpha}$ is non-empty and compact (see  \cite{Bangert1988} for the proof).
For $\alpha \in \Q$, we define periodicity in the following:

\begin{defi}[periodic configurations]
\label{def:periodic}
For $q \in \N$ and $p \in \Z$, a configuration $ x=(x_i)_{i \in \Z}$ is said to be  $(q,p)$-periodic
if $x=(x_i)_{i \in \Z} \in \R^\Z$ satisfies:
\[
x_{i+q}=x_{i} + p,
\]
for any $i \in \Z$.
\end{defi}
It is easily seen that if $x$ is $(q,p)$-periodic, then its rotation number is $p/q$.
This paper discusses only the case where $\alpha \in \Q$.
For $\alpha=p/q \in \Q$, we set:
\[
\calm_{\alpha}^\per := \{ x \in \calm_{\alpha} \mid  \text{$x$ is $(q,p)$-periodic}\}.
\]
\begin{defi}[neighboring pair]
For a set $A \subset \R^\Z$ and $a,b \in A$ with $a<b$, we call $(a,b)$ a neighboring pair of $A$ if $a,b \in A$
and there is no other $x \in A$ with $ a<x<b$. 
Here $a<b$ means $a_i < b_i$ for any $i \in \Z$.
\end{defi}

%Let a map $p_0 \colon \R^\Z \to \R$ by $p_0((x_i)_{i \in \Z})=x_0$.
Given a neighboring pair $(x^0,x^1)$ of $\calm_{\alpha}^\per$, define:
\begin{align*}
\calm_{\alpha}^+(x^0,x^1)&=\{x \in \calm_{\alpha} \mid |x_i - x_i^0| \to 0 \ (i \to -\infty) \ and \ |x_i - x_i^1| \to 0 \ (i \to \infty) \} \ and \\
\calm_{\alpha}^-(x^0,x^1)&=\{x \in \calm_{\alpha} \mid |x_i - x_i^0| \to 0 \ (i \to \infty) \ and \ |x_i - x_i^1| \to 0 \ (i \to -\infty) \}.
\end{align*}
Bangert showed the following proposition.
\begin{prop}[\cite{Bangert1988}]
Given $\alpha \in \Q$, $\calm_{\alpha}^\per$ is nonempty.
Moreover, 
if $\calm_\alpha^\per$ has a neighboring pair, 
then $\calm_{\alpha}^+$ and $ \calm_{\alpha}^-$ are nonempty.
\end{prop}

Although we have discussed minimal configurations in the preceding paragraph, 
there are also interesting works that treat non-minimal orbits between periodic orbits,
particularly,  \cite{Rabinowitz2008} and \cite{Yu2022}.
In  \cite{Rabinowitz2008}, Rabinowitz used minimizing methods to prove the existence of three types of solutions---periodic, heteroclinic and homoclinic---in potential systems with reversibility for time, i.e. $V(t,x)=V(-t,x)$.
Under an assumption called  a {\it gap}, which is similar to a {neighboring pair}, for a set of periodic and heteroclinic solutions, 
non-minimal heteroclinic and homoclinic orbits can be given between two periodic orbits.

Since the heteroclinic/homoclinic orbit is a trajectory that transits between two equilibrium points, we refer to these orbits as {\textit{$k$-transition orbits}} in this paper. For example, a monotone heteroclinic orbit is a one-transition orbit.
Non-minimal orbits are realized as {\it{n-transition orbits}} for $n \ge 2$ and
 they are heteroclinic when $n$ is odd and homoclinic when $n$ is even.
 (You can regard each configuration $x \in \R^\Z$ as its graph $\{(i,x_i) \mid i \in \Z\}$, see the following figures.)

\begin{figure}[htbp]
  \begin{minipage}[b]{0.48\linewidth}
    \centering
    \includegraphics[width=5cm]{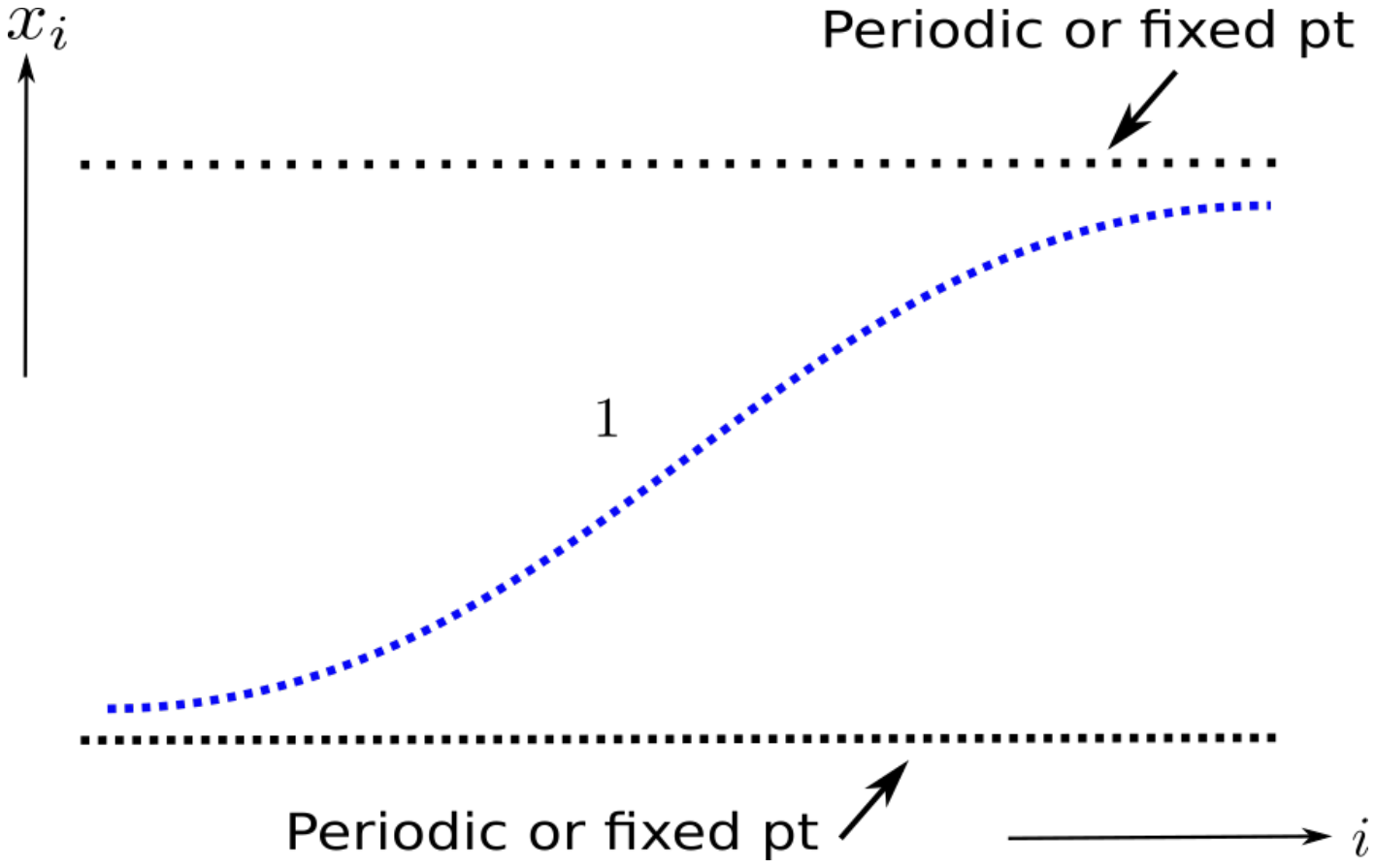}
    \subcaption{$1$-transition heteroclinic orbit}
  \end{minipage}
  \begin{minipage}[b]{0.48\linewidth}
    \centering
    \includegraphics[width=5cm]{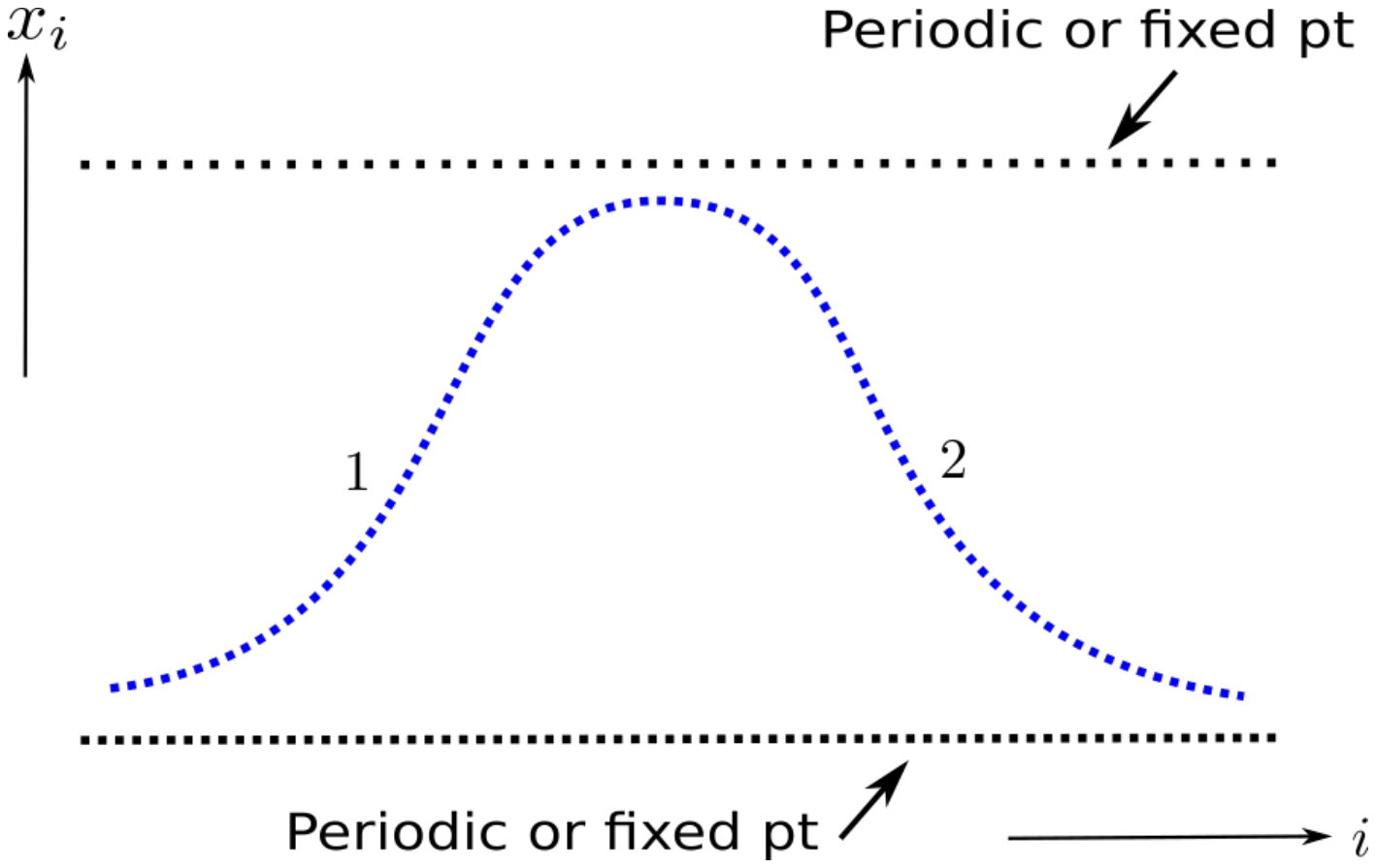}
    \subcaption{$2$-transition homoclinic orbit}
  \end{minipage}
  \caption{One and two transition orbits}
\end{figure}

 \begin{rem}
There are two remarks about one-transition orbits:
(a) Each element of the sets $\calm_{\alpha}^\pm$ implies monotone heteroclinic orbits and Mather's result \cite{Mather1993} is the first to discuss this problem.
(b)
The existence of one-transition orbits ( i.e., monotone heteroclinic orbits)  does not require gaps for heteroclinic orbits.
This can also be illustrated by considering a simple pendulum system
since a set of heteroclinic orbits is dense in its system.
%(b) Definition \ref{defi:transition} and `$k$-transition' in \cite{RabiStre2011} are different. 
%The latter means a solution that passes between $k-1$ periodic orbits in PDEs.
\end{rem}

Rabinowitz's approach can be applied to variational methods for area-preserving maps. 
Yu \cite{Yu2022} added variational principle $h$ to the following assumption $(h_5)-(h_6)$ to $h$:
\begin{itemize}
  \item[($h_5$)] There exists a positive continuous function $p$ on $\R^2$ such that:
  \[
  h(\xi,\eta') +  h(\eta,\xi') -   h(\xi,\xi') -  h(\eta,\eta')  > \int_{\xi}^{\eta}  \int_{\xi'}^{\eta'} p
  \]
  if $\xi < \eta$ and $\xi' < \eta'$.
  \item[($h_6$)] There is a $\theta > 0$ satisfying the following conditions:
  \begin{itemize}
  \item $\xi \mapsto \theta \xi^2 /2 - h(\xi,\xi')$ is convex for any $\xi'$, and
  \item $\xi' \mapsto \theta \xi'^2 /2 - h(\xi,\xi')$ is convex for any $\xi$.
  \end{itemize}
\end{itemize}
In the rest of this paper, we assume $(h_1)-(h_6)$ for $h$.
\begin{rem}
\,
\label{rem:delta_bangert}
\begin{itemize}
\item[(a)] One of a sufficient conditions for $(h_2)-(h_5)$ is 
\[
(\tilde{h}) \ \text{$h \in C^2$ and $\partial_1 \partial_2 h \le -\delta <0$ for some $\delta>0$}.
\]
Bangert \cite{Bangert1988} shows that assuming  $(h_2)-(h_4)$ implies $(\tilde{h})$.
To verify the assumption $(h_5)$, we can ensure it by choosing a positive function $\rho = \delta$.
If a monotone twist map $f$ is of class $C^1$ and satisfies $\partial X/\partial y \ge \delta$ for some $\delta>0$, a generating function $h$ for $f$ satisfies $(\tilde{h})$.
However, $(\tilde{h})$ is not a necessary condition for satisfying $(h_2)-(h_5)$.
%as seen in Section \ref{sec:example_h}.
\item[(b)] Assuming $(h_6)$ allows us to derive Lipschitz continuity for $h$ in the following meaning: there is a Lipschitz constant $C$ satisfying:
\begin{align*}
&h(\xi,\eta_1)-h(\xi,\eta_2) \le C|\eta_1-\eta_2|, \text{and}\\
&h(\xi_1,\eta)-h(\xi_2,\eta) \le C|\xi_1-\xi_2|
\end{align*}
%In fact,]
\item[(c)] 
If $h$ is of class $C^1$, we do not require $(h_6)$.
\end{itemize}
\end{rem}
 
 Clearly, $(h_5)$ implies $(h_3)$.
Mather \cite{Mather1987} proved
that  if $h$ satisfies $(h_1)$-$(h_6)$, then $\partial_2 h (x_{i-1},x_i)$ and
 $\partial_1h(x_i,x_{i+1})$ exist in the meaning of the left-sided limit
 (even if $h$ is not differentiable).
 In addition, he proved that if $x$ is a locally minimal configuration, then it satisfies
 $\eqref{stationary}$.
 Hence we can treat a stationary configuration for non differentiable functions.

Yu applied Rabinowitz's methods to monotone twist maps to show finite transition orbits of monotone twist maps for all $\alpha=p/q \in \Q$.
We will give a summary of his idea in the case of $\alpha=0$ (i.e. $(q,p)=(1,0)$ in Definition \ref{def:periodic}).
 Let $(u^0,u^1)$ be a neighboring pair of $\calm_0^\per$.
 By abuse of notation, we then denote $u^j$ for $j=0,1$ by the constant configuration $u^j= (x_i)_{i\in\Z}$ where $x_i=u^j$ for all $i \in \Z$.
We set:
\begin{align}
\label{c}
{c:=\min_{ x \in \R} h(x,x)}(=h(u^0,u^0)=h(u^1,u^1)).
\end{align}
And:
\begin{align}
\label{pre-action}
I(x) := \sum_{i \in \Z } a_i(x),
\end{align}
where $a_i(x)= h(x_i,x_{i+1}) - c$.
Yu \cite{Yu2022} studied local minimizers of $I$ to show the existence of finite transition orbits,
i.e., heteroclinic or homoclinic orbits.
Given a rational number $\alpha \in \Q$ and a neighboring pair  $(x^0,x^1)$ of $\calm_{\alpha}^\per$,
we let:
\begin{align*}
I^+_{\alpha}(x^0,x^1)&=\{x_0 \in \R \mid x=(x_i)_{i \in \Z} \in \calm_{\alpha}^+(u^0,u^1)\}, \ and\\
I^-_{\alpha}(x^0,x^1)&=\{x_0 \in \R \mid x=(x_i)_{i \in \Z} \in  \calm_{\alpha}^-(u^0,u^1) \}.
\end{align*}
%where $\calm_{\alpha}^0(u^0,u^1)$ is a subset of $\calm_{\alpha}^+(u^0,u^1)$ and
%$\calm_{\alpha}^1(u^0,u^1)$ is also a subset of $\calm_{\alpha}^-(u^0,u^1)$.
%We give each precise definition of $\calm_{\alpha}^i(u^0,u^1)$ $(i=0,1)$ in Section \ref{sec:pre}.
Under the above setting, he showed:
\begin{theo}[Theorem 1.7, \cite{Yu2022}]
\label{theo:yu_main}
Given a rational number $\alpha \in \Q$ and a neighboring pair $(x^0,x^1)$ of $\calm_{\alpha}^\per$.
If
\begin{align}
\label{hetero_gap}
I^+_{\alpha}(x^0,x^1) \neq (x_0^0,x_0^1) \ \text{and} \
I^-_{\alpha}(x^0,x^1) \neq (x_0^0,x_0^1) ,
\end{align}
then,
for every ${\delta}>0$ small enough,
there is an $m=m(\delta)$
such that for every sequence of integers $q=( q_i )_{i \in \Z}$ with $q_{i+1}-q_{i} \ge 4m$ and for every $j,k \in \Z$ with $j<k$,
there is a configuration $x=(x_i)_{i \in \Z} $ for $h$ satisfying:
\begin{enumerate}
\item $x_i^0<x_i<x_{i}^1$ for all $i \in \Z$;
\item $|x_{q_i-m}-x^i_{q_i-m}| \le \delta$ \text{and} \, $|x_{q_i+m}-x^i_{q_i+m}| \le \delta$ for all $i=j,\dots,k$;
\item $|x_i - x_i^j| \to 0$ as $i \to -\infty$ \text{and} \, $|x_i - x_i^k| \to 0$ as $i \to +\infty$.
\end{enumerate}
Here, for any $j \in \Z$, $x^j=x^0$, if $j$ is even, and $x^j=x^1$, if $j$ is odd.
\end{theo}

Furthermore, Rabinowitz \cite{Rabinowitz2008} proved the existence of an infinite transition orbit as a limit of sequences of finite transition orbits.
However, the variational structure of infinite transition orbits for potential systems is an open question in his paper.
To consider the question for twist maps, the following proposition is crucial.
\begin{prop}[Proposition 2.2, \cite{Yu2022}]
\label{prop:finite}
If $I(x) < \infty$, then $|x_i -u^1| \to 0$ or $|x_i -u^0| \to 0$ as $|i| \to \infty$.
\end{prop}
Since this implies that $I(x) = \infty$  if $x$ has infinite transitions, 
we need to fix the normalization of $I$.
Therefore, we focus on giving the variational structure and boundary condition that characterize infinite transition orbits of monotone twist maps.
As a result, the function $J$ and set  $\xkr$ defined in Section \ref{sec:pf} represent a variational structure and a configuration space for infinite transition orbits.
Through this variational problem, we showed:
\begin{theo}[Our main theorem]
\label{theo:main}
Assume the same condition of Theorem \ref{theo:yu_main}.
Then, for every positive sequence ${\epsilon}=(\epsilon_i)_{i \in \Z}$,
%with $\epsilon_i$ small enough,
there is an $m=( m_i )_{i \in \Z}$
such that for every sequence of integers $k=(k_i)_{i \in \Z}$ with $k_{i+1}-k_{i} \ge m_i$,
there is a configuration $x=(x_i)_{i \in \Z}$ for $h$ satisfying:
\begin{enumerate}
\item $x_i^0<x_i<x_{i}^1$ for all $i \in \Z$;
\item  for any $j \in \Z$, \ $|x_{i}-x^{2j}_{i}| \le \epsilon_{2j}$  if $i \in [k_{4j}, k_{4j+1}]$ \text{and} \, 
$|x_{i}-x^{2j+1}_{i}| \le \epsilon_{2j+1}$ if $i \in [k_{4j+2}, k_{4j+3}]$.
\end{enumerate}
%For any $j \in \Z$, $x^j=x^0$, if $j$ is even, and $x^j=x^1$, if $j$ is odd.
%Under the same assumption as Theorem \ref{theo:yu_main},
%there exist uncountably infinitely many $\infty$-transition orbits.
\end{theo}

%We give some remarks.
%In Theorem \ref{theo:main_mtm},  two distinct fixed points on $\R \times [a,b]$ may coincide when viewed as points on $\T \times [a,b]$.

This paper is organized as follows.
Section \ref{sec:pre} deals with Yu's results in \cite{Yu2022} and related remarks.
In Section \ref{sec:pf}, our main results are stated.
We first give the proof of the case of $\alpha=0$ and
then see the generalized cases.
Section \ref{sec:add} provides, as additional discussions,
a special example and the estimate of the number of the obtained infinite transition orbits.
%Section \ref{sec:example_h} provides the special case of $h$.
%In the last section, we extend the result of the infinite transition orbits to a more general representation.
%In the fourth section, we give interesting examples.

%\newpage

\section{Preliminary}
\label{sec:pre}
In this section, we would like to introduce properties of $\eqref{pre-action}$ and minimal configurations using several useful results in \cite{Yu2022}.
%Though we do not use the assumption $(h_1)-(h_6)$ directly for our proof in Section \ref{sec:pf},they are needed because we use lemmata and propositions in \cite{Yu2022}.
Moreover, we study estimates of monotone heteroclinic orbits \ (1-transition orbits).

\subsection{Properties of minimal configurations}
Let $(u^0,u^1)$ be a neighboring pair of $\calm_0^\per$ and:
\begin{align}
\label{eq:x}
\begin{split}
X &= X(u^0,u^1)= \{ x = (x_i)_{i \in \Z} \mid u^0 \le x_i \le u^1 \ ({}^{\forall} i \in \Z) \},\\
X(n)&=X(n ; u^0,u^1)=\{ x =(x_i)_{i=0}^{n}  \mid u^0 \le x_i \le u^1 \ ({}^{\forall} i \in \{0, \cdots, n\}) \}, and\\
\hat{X}(n) &= \hat{X}(n ; u^0,u^1)=\{ x =(x_i)_{i=0}^{n}   \mid x_0=x_n, \ u^0 \le x_i \le u^1 \ ({}^{\forall} i \in \{0, \cdots, n\}) \}.
\end{split}
\end{align}

\begin{defi}[\cite{Yu2022}]
For $x \in X$, we set:
\[d(x):=\max_{0 \le i \le n} \min_{j \in \{0,1\}} |x_i -u^j|.\]
For any $\delta > 0$, let:
\begin{align}
\label{phi}
\phi(\delta):= \inf_{n \in \Z^+}
\inf \left\{
\sum_{i=0}^{n-1} a_i(x) \mid x \in \hat{X}(n) \ {\text{and}} \ d(x) \ge \delta
\right\}.
\end{align}
\end{defi}

\begin{lemm}[Lemma 2.7, \cite{Yu2022}]
\label{lemm:yu_lower}
The function $\phi$ is continuous and  satisfies $\phi(\delta)>0$ if $\delta >0$; $\phi(\delta)=0$ if $\delta=0$.
It increases monotonically with respect to $\delta$.
Moreover,
for any $n \in \N$ and $x \in \hat{X}(n)$ satisfying
\[
\min_{j=0,1}  |x_i -u^j| \ge \delta,\ (i=1,\cdots,n-1),
\] 
then
\[\sum_{i=0}^{n-1} a_i(x) \ge n \phi(\delta)\]
and for any $n \in \N$ and $x \in {X}(n)$,
\end{lemm}

\begin{lemm}[Lemma 2.8, \cite{Yu2022}]
\label{lemm:finite_bdd}
For any $n \in \N$ and $x \in {X}(n)$ satisfying $d(x) \ge \delta$,
\[
\sum_{i=0}^{n-1} a_i(x) \ge \phi(\delta)-C|x_n-x_0| \ge -C|x_n-x_0|.\]
\end{lemm}
\begin{proof}
See \cite{Yu2022}. This proof requires $(h_3)$.
\end{proof}

\begin{lemm}[Lemma 2.10, \cite{Yu2022}]
\label{lemm:positive}
If $x \in X$ satisfies  $|x_i - u^0|$ (resp. $|x_i - u^1|$) as $|i| \to \infty$ and $x_i \neq u^0$ (resp. $x_i \neq u^1$) for some $i \in \Z$,
then $I(x)>0$.
\end{lemm}
In using a minimizing method to get a stationary configuration,
we need to check that each component of the obtained minimizer is not equal to $u^0 $ or $u^1$.
This follows from the next lemmas.
\begin{lemm}[Lemma 2.11, \cite{Yu2022}]
\label{lemm:estimate_periodic1}
For any $\delta \in (0, u^1-u^0]$, if $(x_i)_{i=0}^{2}$ satisfies:
\begin{enumerate}
\item $x_i \in [u^0,u^1]$ for all $i=0,1,2$;
\item $x_1 \in [u^1-\delta,u^1]$, and $x_0 \neq u^1$ or $x_2 \neq u^1$; and
\item $h(x_0,x_1,x_2) \le h(x_0, \xi,x_2)$ for all $\xi \in [u^1-\delta,u^1] $,
\end{enumerate}
then $x_1 \neq u^1$.
This still holds
if we replace every $u^1$ by $u^0$ and every $[u^1-\delta,u^1]$ by $[u^0, u^0+\delta]$.
\end{lemm}

\begin{lemm}[Lemma 2.12, \cite{Yu2022}]
\label{lemm:estimate_periodic2}
For any $n_0$ and $n_1 \in \N$ with $n_0 < n_1$, if a finite configuration $x=(x_i)_{i=n_0}^{n_1}$ satisfies:
\begin{enumerate}
\item $x_i \in [u^0,u^1]$ for all $i=n_0, \cdots, n_1$ and
\item for any $(y_{i})_{i=n_0}^{n_1}$ satisfying $y_{n_0}=x_{n_0}$, $y_{n_1}=x_{n_1}$, and $y_i \in [u^0,u^1]$,
\[
h(x_{n_0},x_{n_0+1}, \cdots, x_{n_1-1}, x_{n_1}) \le h(y_{n_0},y_{n_0+1}, \cdots, y_{n_1-1}, y_{n_1}),
\]
\end{enumerate}
then $x$ is a minimal configuration.
Moreover, if $x$ also satisfies $x_{n_0} \notin \{u^0,u^1\}$ or  $x_{n_1} \notin \{u^0,u^1\}$,
then $x_i \notin \{u^0,u^1\} $ for all $i=n_{0}+1, \cdots, n_1-1$.
\end{lemm}
\begin{proof}[Proof of the two lemmas above]
See \cite{Yu2022}. These proofs require $(h_4)$ and $(h_5)$.
\end{proof}

Moreover, we can replace $\alpha=0$ with arbitrarily other rational numbers as seen below.
\begin{defi}[Definition 5.1, \cite{Yu2022}]
For $\alpha=p/q \in \Q \backslash \{0\}$, we set:
\begin{align*}
%X_{\alpha}(q;x^-,x^+)&:=\{x=(x_i)_{i=0}^{q} \mid x_i^- \le x_i \le x_i^+ (i=0, \cdots ,q)\}\\
%X_{\alpha}(q;x^-,x^+)&:=\{x \in X(q) \mid x_q =x_0 + p\}\\
X_{\alpha}(x^-,x^+)&:=\{x=(x_i)_{i \in \Z} \mid x_i^- \le x_i \le x_i^+ (i \in \Z)\}.
\end{align*}
where $x^-$ and $x^+$ is in $\calm_{\alpha}^{\per}$ and $(x^-_0,x^+_0)$ is a  neighboring pair in $\calm_{\alpha}^{\per}$.
\end{defi}

\begin{defi}[Definition 5.2, \cite{Yu2022}]
Let $h_i \colon \R^2 \to \R$ be a continuous function for $i=1,2$.
For $h_1$ and $h_2$, we define $h_1 \ast h_2 \colon \R^2 \to \R$ by
\[
h_1 \ast h_2(x_1,x_2) = \min_{\xi \in \R} (h_1(x_1, \xi) + h_2(\xi, x_2)).
\]
We call this the \it{conjunction} of $h_1$ and $h_2$.
\end{defi}

Using the conjunction, we define a function $H \colon \R^2 \to \R$ for $\alpha=p/q$  by:
\[
H(\xi,\xi') = h^{*q}(\xi,\xi'+p),
\]
where
$h^{*q}(x,y)=h_1 \ast h_2 \ast \cdots \ast h_q(x,y)$ and 
 $h_i=h$ for all $i = 1,2, \cdots , q$.

\begin{defi}[Definition 5.5, \cite{Yu2022}]
For any $y =(y_i)_{i \in \Z} \in X(x_0^-,x_0^+)$,
we define $x=(x_i)_{i \in \Z} \in X_{\alpha}(x^-,x^+)$ as follows:
\begin{enumerate}
\item $x_{iq}=y_i + ip$ and
\item $(x_j)_{j=iq}^{(i+1)q}$ satisfies
\[
h(x_{iq}, \cdots, x_{(i+1)q}) =H(x_{iq},x_{(i+1)q}) = H(y_i,y_{i+1}),
\]
i.e.,  $(x_j)_{j=iq}^{(i+1)q}$ is a minimal configuration of $h$.
\end{enumerate}
\end{defi}
Although we focus on the case of rotation number $\alpha=0$, we may apply our proof to all rational rotation numbers from the following.
%(We see the details in the next section.)
\begin{prop}[Proposition 5.6, \cite{Yu2022}]
\label{prop:alpha_configuration}
Let $y \in X(x_0^-,x_0^+)$ and $x \in X_{\alpha}(x^-,x^+)$ be defined as above.
If $y$ is a stationary configuration of $H$, 
then $x$ must be a stationary configuration of $h$.
\end{prop}

\subsection{Some remarks for heteroclinic orbits}

Let $X^0$ and $X^1$ be given by:
\begin{align*}
X^0&=\{x \in X \mid |x_i - u^1| \to 0 \ (i \to \infty), |x_i - u^0| \to 0 \ (i \to -\infty)\} \ and\\
X^1&=\{x \in X \mid |x_i - u^1| \to 0 \ (i \to -\infty), |x_i - u^0| \to 0 \ (i \to \infty)\}.
\end{align*}
By considering a local minimizer (precisely, a global minimizer in $X^0$ or $X^1$),
Yu \cite{Yu2022} proved the existence of heteroclinic orbits, which Bangert showed in \cite{Bangert1988},
as per the following proposition.
\begin{prop}[Theorem 3.4 and Proposition 3.5, \cite{Yu2022}]
There exists a stationary configuration $x $ in $X^0$ (resp. $X^1$) satisfying $I(x)=c_0$ (resp. $I(x)=c_1$),
where
\[c_0=\inf_{x \in X^0} I(x), \ c_1=\inf_{x \in X^1} I(x)\]
Moreover, $x$ is strictly monotone, i.e., $x_i<x_{i+1}$ (resp. $x_i>x_{i+1}$)  for all $i \in \Z$.
\end{prop}
Let:
\begin{align*}
&\calm^0(u^0,u^1)=\{x \in X \mid c_0=\inf_{x \in X^0} I(x)\} \ and\\
&\calm^1(u^0,u^1)=\{x \in X \mid c_1=\inf_{x \in X^1} I(x)\}.
\end{align*}

Set
\begin{align}
\label{cstar}
c_\ast:=I(x^0) + I(x^1),
\end{align}
where $x^i \in \calm^i(u^0,u^1)$ $(i=0,1)$.
From the above and Lemma \ref{lemm:positive}, we immediately obtain the following corollary.
\begin{co}
\label{co:positive_het}
$c_{\ast}>0$
\end{co}
\begin{proof}
Choose  $x^0 \in \calm^0(u^0,u^1)$ and $x^1 \in \calm^1(u^0,u^1)$ arbitrarily.
From monotonicity, $x^0$ and $x^1$ intersect exactly once.
We define $x^+$ and $x^-$ in $X$ by
$x^+_i := \max\{x^0_i,x^1_i\}$ and $x^-_i := \min\{x^0_i,x^1_i\}$.
By $(h_3)$ and Lemma \ref{lemm:positive},
\[
c_{\ast}=I(x^0) + I(x^1) \ge I(x^+) + I(x^-) >0.
\]
This completes the proof.
\end{proof}

\begin{lemm}
\label{lemm:finite_het}
For any $\epsilon >0$, there exist $n_0 \in \N$ and $x \in \calm^0(u^0,u^1)$ (resp. $x \in \calm^1(u^0,u^1))$ such that $\sum_{i=-n}^{n-1} a_i(x) \in (c_0 -\epsilon, c_0 + \epsilon)$ $(resp. \sum_{i=-n}^{n-1} a_i(x) \in (c_1 -\epsilon, c_1 + \epsilon))$ for all $n \ge n_0$.
\end{lemm}
\begin{proof}
For sufficiently large  $n_0$,
there exists $y \in \calm^0$ such that for any $n \ge n_0$,
\[
y_{-n}-u^0 < \epsilon/2C, \text{and} \ u^1-y_n< \epsilon/2C.
\]
Since $c_0=\sum_{i \in \Z} a_i(y)$
by the minimality of $y$,
we get:
\begin{align*}
\left| \sum_{i=-n}^{n-1} a_i(y) -c_0 \right|
= \left| \sum_{i <-n} a_i(y) +  \sum_{i \ge n} a_i(y) \right|
\le C( (y_{-n}-u^0) + (u^1-y_n))
<  \epsilon
\end{align*}
as desired.
A similar way is valid for the rest of the proof.
\end{proof}
We will check the properties of the `pseudo' minimal heteroclinic orbits.
Under the assumption $\eqref{hetero_gap}$, the following lemma holds.
\begin{lemm}[Proposition 4.1, \cite{Yu2022}]
\label{lemm:gap_estimate}
Assume $\eqref{hetero_gap}$ holds.
For any $\epsilon>0$, there exist $\delta_i \in (0, \epsilon)$ $(i=1,2,3,4)$ and positive constants  $e_0=e_0(\delta_1,\delta_2)$ and $e_1=e_1(\delta_3,\delta_4)$ satisfying:
\begin{align*}
\inf\{ I(x) \mid x \in X^0, x_0 =u_0+\delta_1  \text{ or }  x_0=u_1-\delta_2\} &= c_0+e_0 \ and\\
\inf\{ I(x) \mid x \in X^1, x_0 =u_1-\delta_3  \text{ or }  x_0=u_0+\delta_4\} &= c_1+e_1.
\end{align*}
\end{lemm}
We omit the proof.
As a result, we need to choose each $\delta_i$ small enough satisfying $\delta_1, \delta_2 \in I^+_{0}(u^0,u^1)$ and $\delta_3, \delta_4 \in I^-_{0}(u^0,u^1)$.
It is immediately shown that:
\begin{lemm}
\label{lemm:finite_fake_het}
Let $x \in X^0$ (resp. $x \in X^1)$ be satisfy $I(x)=c_0+e_0$ (resp. $I(x)=c_1+e_1$).
Then,
for any $\epsilon >0$, there exist $n_0 \in \Z_{\ge 0}$ such that $\sum_{i=-n}^{n-1} a_i(x) \in (c_0+e_0 -\epsilon, c_0+e_0 + \epsilon)$
$(resp. \sum_{i=-n}^{n-1} a_i(x) \in (c_1+e_1 -\epsilon, c_1+e_1 + \epsilon))$ for all $n \ge n_0$.
\end{lemm}

%\newpage

\section{The proofs of our main theorem and some remarks}
\label{sec:pf}
\subsection{Variational settings}
\label{subsec:setting}
%Set  $\tilde{I}_0:=(u_0,u_1) \backslash I_0$ and  $\tilde{I}_1:=(u_0,u_1) \backslash I_1$.
%We assume a gap condition for periodic and heteroclinic configurations, i.e., $\tilde{I}_0$ and $\tilde{I}_1$ are nonempty sets.
Let $(u^0,u^1)$ be a neighboring pair of $\calm^\per_{0}$ and set:
\begin{align}
\label{def:kp}
\begin{split}
K&=\left\{ k=(k_i)_{i \in \Z} \subset \Z \mid k_0=0,  k_i < k_{i+1} \right\}, and:\\
\calc(n;a,b)&=\min  \left\{ \sum_{i=0}^{n-1} h(x_i,x_{i+1}) \mid {x \in X(n),x_0=a,x_n=b} \right\}.
\end{split}
\end{align}
 (See $\eqref{eq:x}$ for the definition of $X$.)
For $k \in K$, 
set $I_i = [k_i, k_{i+1}-1] \cap \Z$
and
 $|I_i|=|k_i-k_{i+1}|$
 for each $i \in \Z$.
 Now we define the renormalized function $J$ by:
\[
J(x)=J_k(x)=\sum_{j \in \Z} A_j(x),
\]
where
$A_j(x)=h(x_j,x_{j+1}) -c(j)$ and:
\begin{align*}
c(j) =
\begin{dcases}
%\calc(|I_{2i}|,u^i,u^i)/|I_{2i}|& \ (j \in I_{2i} )\\
\frac{\calc(|I_{2i+1}|,u^i,u^{i+1})}{|I_{2i+1}|}&\ (j \in I_{2i+1} \ \text{for some} \ i \in \Z )\\
\frac{\calc(|I_{2i}|,u^i,u^i)}{|I_{2i}|}  &\ (\text{otherwise})
\end{dcases}
.
\end{align*}

\begin{rem}
$(a)$ 
The existence of the minimum value $\calc(n;a,b)$ is guaranteed by $(h_2)$.
$(b)$ Theorem 5.1  in \cite{Bangert1988}
shows that $x \in \calm_{\alpha}^{\per}$ has minimal period $(q,p)$ with $q$ and $p$ relatively prime.
It indicates that:
\[
c =\frac{\calc(|I_{2i}|,u^0,u^0)}{|I_{2i}|}=
\frac{\calc(|I_{2i}|,u^1,u^1)}{|I_{2i}|}.
\]
\end{rem}
Next, we set:
\begin{align*}
P&=\left\{ \rho=(\rho_i)_{i \in \Z} \subset \R_{>0} \mid 0 <\rho_i <\frac{u^1-u^0}{2} \ ({}^{\forall} i \in \Z), \  \sum_{i \in \Z} \rho_i <\infty  \right\}.
\end{align*}
For $k \in K$ and $\rho \in P$, the set $\xkr$ is given by:
\[
%\xkr =\{ x \in \R^\Z  \mid  \text{$0 \le  x_{k_i} - u_{k_i}^0 \le \rho_i$ if $i \equiv 0, 1$  and  $0 \le  u_{k_i}^1 - x_{k_i} \le \rho_i$ if $i \equiv -1, 2$} \}
\xkr =
 \bigcap_{i \in \Z}
 \left\{
\left( \bigcap_{i \equiv 0,1}  Y^0(k_i,\rho_i) \right) \cap \left( \bigcap_{i \equiv -1,2}  Y^1(k_i,\rho_i) \right)
\right\}
,
\]
where
\[
Y^j(l,p) = \{x \in X \mid |x_l - u^j| \le p \} \ (j=0,1)
\]
 and $ a \equiv b$ means $a \equiv b \ (\mathrm{mod} \ 4)$.
 (See $\eqref{eq:x}$ for the definition of $X(n)$.)
It is easily seen that each element of $\xkr$ has infinite transitions.
\begin{figure}[htbp]
    \centering
    \includegraphics[width=8.5cm, angle=-90]{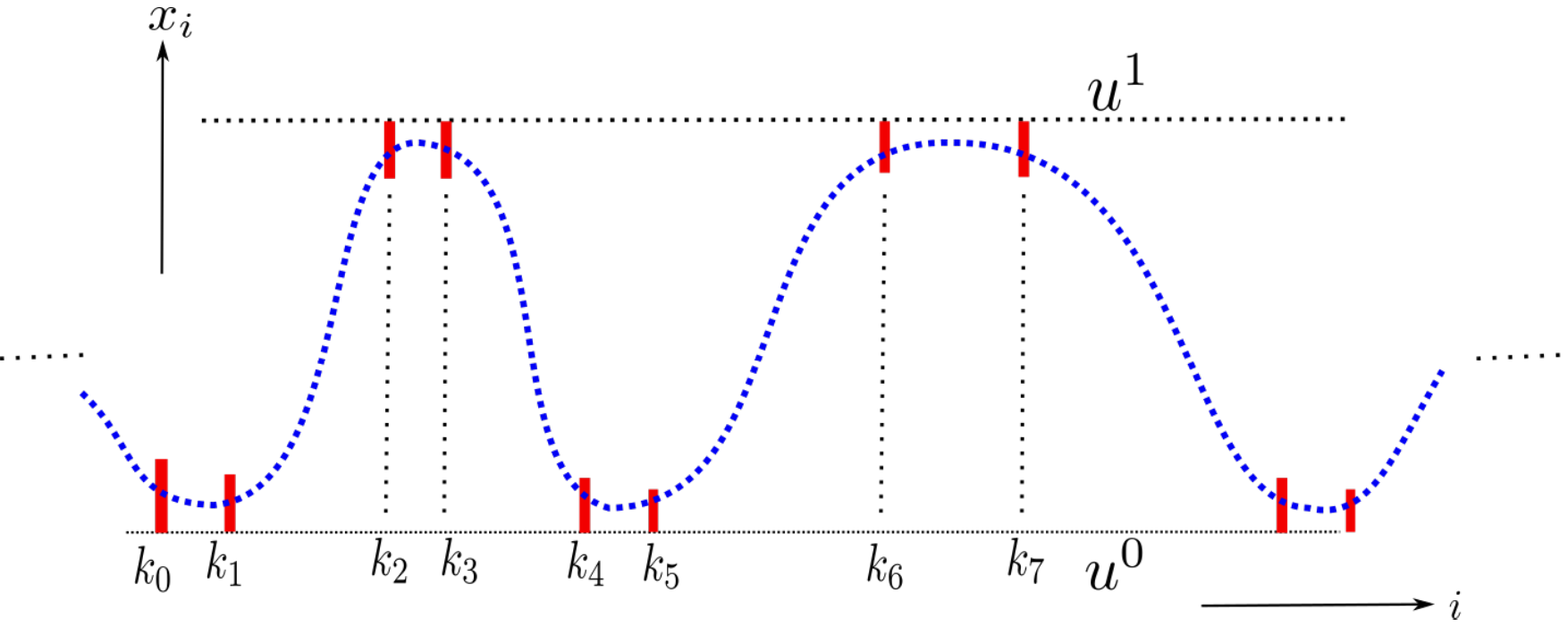}
  \caption{An element of $\xkr$}
\end{figure}
Notice that since compactness and sequential compactness are equivalent in the presence of the second countability axiom,
$X$ is a  sequentially compact set by Tychonoff's theorem.
Clearly, $\xkr$ is a closed subset of $X$, so the set $\xkr$ is also sequentially compact.

\if0
\begin{lemm}
\label{lemm:cpt}
The set $X$  is sequentially compact.
\end{lemm}
\begin{proof}
By Tychonoff's theorem, $X$ is a compact set.
Since a compact set on a metric space is sequentially compact,
we only need to define the metric on $X$ as
\begin{align*}
d(x,y) &= \sum_{i \in \Z} \frac{|x_i - y_i|}{2^{|i|}}.
\end{align*}

It suffices to check that $X$ is metrizable.
Let $d \colon X \times X  \to \R$ be given by
\begin{align*}
d(x,y) &= \sum_{i \in \Z} \frac{|x_i - y_i|}{2^{|i|}}
\end{align*}
Clearly, $d$ is a metric function.
We show that convergence on  $d$ and  $\eqref{t_metric}$ is equivalent.
Since for all $i \in \Z$
\[
\frac{|x_i - y_i|}{2^{|i|}} \le  d(x,y),
\]
it is sufficient to show that for each $x,y \in X$, the function $d(x,y)$ goes to $0$ if $\eqref{t_metric}$ holds.
Let $(x^n)$ be a convergence sequence to $y$.
There is a constant $M >0$ such that for all $j \in \Z$ 
\begin{align*}
d(x^{n},y) \le  c(j,n)+ \frac{M}{2^{|j|}}
\end{align*}
where
$c(j,n) =  \sum_{i \le |j|} {|x_i^{n} - y_i|}/{2^{|i|}}$.
Notice that for each $j \in \Z$, $c(j,n) \to 0$ as $n \to \infty$.
Thus, for any $\epsilon > 0$, we can take $i_0$ and $n_0$ such that
$ {M}/{2^{|i_0|}}< {\epsilon}/{2}$ and  $c(i_0,n_0)  < {\epsilon}/{2}$,
thus completing the proof.

\end{proof}
\fi

%Notice that $(h_1)$ implies that the values of $c^{\pm}_i$ depend on $\rho_i, \rho_{i+1}$ and $k_{i+1}-k_i$.

As a basic property of $J$, 
we first show that $J(x)$ can be finite unlike $I(x)$ even if $x$ is an infinite transition orbit.
\begin{lemm}
\label{lemm:upper}
If $\rho \in P$, then there exists $y=(y_i)_{i \in \Z} \in \xkr$ such that $J(y) =0$ for all $k \in K$.
\end{lemm}
\begin{proof}
By the definition of $J$, we can choose a configuration $y \in \xkr$ satisfying
 $\sum_{i=k_j}^{k_{j+1}}A_i(y)=0$ for all $j \in \Z$ by
 taking $y$ such that it satisfies $\sum_{i=k_j}^{k_{j+1}}h(y_i,y_{i+1})=\calc(|I_{2j+1}|,u^j,u^{j+1})$
 or $\sum_{i=k_j}^{k_{j+1}}h(y_i,y_{i+1})=\calc(|I_{2j}|,u^j,u^{j})$.
\end{proof}

The above lemma implies that $J$ overcomes the problem referred to in Proposition \ref{prop:finite}.
Next, we show that $J$ is bounded below.

\begin{lemm}
\label{lemm:lower}
If $\rho \in P$, then there is a constant $M \in \R$ such that $J(x) \ge M (> -\infty)$ for all $x \in \xkr$.
\end{lemm}
\begin{proof}
For each $x \in \xkr$ with $J(x) \le 0$, we define $y=(y_j)_{j \in \Z} \in \xkr$ by
 $y_{k_i} = u^0$; if $i \equiv 0,1$, \ $y_{k_i} = u^1$; if $i \equiv -1,2$, and $y_j=x_j$ otherwise.
 From the definition, $0 \le J(y) < \infty$. Lipschitz continuity of $h$ shows:
 \[
-J(x) \le J(y) -J(x) \le 2C \sum_{i \in \Z} \rho_i < \infty
 \]
and we get
$J(x) \ge -2C \sum_{i \in \Z} \rho_i$
for all $x \in \xkr$, thus completing the proof.
\end{proof}

\begin{rem}
In a similar way to the proof in the above lemma, we get
$\displaystyle{
\sum_{i=k_n}^{k_m} A_{i}(x) \ge - 2C \sum_{i=n}^{m} \rho_i}$
for any $n<m$.
\end{rem}

To ensure that $J$ has a minimizer in $\xkr$,
we present the following lemma.

\begin{lemm}
\label{lemm:welldefi}
The function $J$ is  well-defined on $\R \cup \{+\infty\}$, i.e., 
\[
\alpha:=\liminf_{n \to \infty}  \sum_{|i| \le n} A_i(x) = \limsup_{n \to \infty}  \sum_{|i| \le n} A_i(x)=:\beta.
\]
\end{lemm}
\begin{proof}
For the proof, we use a similar argument to Yu's proof of Proposition 2.9 and Lemma 6.1 in \cite{Yu2022}.
By contradiction, we assume $\alpha<\beta$.
First, we consider the case where $\beta = +\infty$.
Fix $\gamma \in \R_{<0}$ arbitrarily.
For $\alpha< +\infty$, we take a constant $\tilde{\alpha}$ with $\tilde{\alpha} > \alpha+1-2\gamma $.
Then there are constants $n_0$ and $n_1$ such that $n_0 < n_1$ and:
\begin{align*}
\sum_{|i| \le n_0} A_i(x) \ge \tilde{\alpha} \ \text{and} \  \sum_{|i| \le n_1} A_i(x) \le \alpha+1.
\end{align*}
Then,
\[
2\gamma > \alpha+1-\tilde{\alpha}\ge
  \sum_{|i| \le n_1} A_i(x) - \sum_{|i| \le n_0} A_i(x) 
=\sum_{ i= -n_1}^{-n_0} A_i(x)  +  \sum_{i = n_0}^{n_1} A_i(x).
\]
Combining the first term and end terms implies:
\[
\sum_{ i= -n_1}^{-n_0} A_i(x) < \gamma \ \text{or} \ \sum_{ i= n_0}^{n_1} A_i(x)  <\gamma.
\]
For $\gamma$ small enough, this contradicts Lemma \ref{lemm:lower}.

Next, we assume $\beta<+\infty$.
Since $\alpha< \beta$, there are two sequences of positive integers $\{m_j \to \infty\}_{j \in \N}$ and $\{l_j \to \infty\}_{j \in \N}$
satisfying $m_j<m_{j+1}$, $l_j<l_{j+1}$ and $m_j+1 < l_j < m_{j+1} -1$ for all $j \in \Z_{>0}$, and:
\[
 \beta= \lim_{j \to \infty} \sum_{i \le |m_j|} A_i(x) >  \lim_{j \to \infty}  \sum_{i \le |l_j|} A_i(x) =\alpha.
\]
Then we can find $j \gg 0$ such that
\begin{align}
\label{a-b}
 \sum_{i \le |l_j|} A_i(x) - \sum_{i \le |m_j|} A_i(x) =  \sum_{i =-l_j}^{-m_j} A_i(x) + \sum_{i =m_j}^{l_j} A_i(x) < \frac{\alpha -\beta}{2}.
\end{align}
Since $|l_j|$ and $|m_j|$ are finite for fixed $j$,
the above calculation does not depend on the order of the sums.
For sufficiently large $j$, a similar argument in the proof of Lemma \ref{lemm:lower} shows:
\[
\sum_{i = -l_j}^{-m_j} A_i(x) \ge -2C\sum_{i = -l_j}^{-m_j} \rho_i > \frac{\alpha-\beta}{4}
\]
and
\[
\sum_{i = m_j}^{l_j} A_i(x) \ge -2C\sum_{i = m_j}^{l_j} \rho_i > \frac{\alpha-\beta}{4}
\]
because  $\rho \in P$ implies
\[
\sum_{|i|>n} \rho_i \to 0 \ (n \to \infty).
\]
and both $m_j$ and $l_j$ goes to infinity as $j \to \infty$.
Therefore:
\begin{align*}
 \sum_{i =-l_j}^{-m_j} A_i(x) + \sum_{i =m_j}^{l_j} A_i(x) > \frac{\alpha-\beta}{2},
 \end{align*}
which contradicts $\eqref{a-b}$.
\end{proof}

%\begin{lemm}
%\label{lemm:semiconti}
%The function $J$ is  semi-continuous on $\xkr$.
%\end{lemm}
\begin{prop}
\label{prop:min}
For all $k \in K$ and $\rho \in P$, there exists a minimizer of $J$ in $\xkr$.
\end{prop}
\begin{proof}
By Lemma \ref{lemm:upper} and \ref{lemm:lower},
we can take a minimizing sequence $x=(x^n)_{n \in \N}$ of $J$ with each $x^n \in \xkr$.
Since $\xkr$ is  sequentially compact, there exists $\tilde{x} \in \xkr$
which  $x^{n_k}$ converges to  $\tilde{x}$ for some subsequence $(n_k)_{k \in \N}$.
Below, we assume $n_k=k$ for simplicity.
To ensure our claim, it is enough to show that
for any $\epsilon>0$, there exists $j_0$ and $n_0 \in \N$
such that:
\begin{align}
\label{eq:lower_conti}
\sum_{|i|>j_0} A_i(x^n) >-\epsilon \ (\text{for all} \ n \ge n_0)
\text{\ and\ }
\sum_{|i|>j_0} A_i(\tilde{x}) <\epsilon,
\end{align}
because if the above inequalities hold, we obtain:
\begin{align*}
J(\tilde{x})
&=\sum_{|i| \le j_0} A_i(\tilde{x})+ \sum_{|i| > j_0} A_i(\tilde{x})\\
&\le \lim_{n \to \infty} \sum_{|i| \le j_0} A_i(x^n) + \epsilon
= \lim_{n \to \infty} (\sum_{i \in \Z} A_i(x^n) - \sum_{|i|>j_0} A_i(x^n))+ \epsilon \\
&\le  \lim_{n \to \infty} \sum_{i \in \Z} A_i(x^n) + 2\epsilon
= \lim_{n \to \infty}  J(x^n) + 2\epsilon.
\end{align*}
Using an arbitrary value of $\epsilon$, we have $J(\tilde{x}) \le  \lim_{n \to \infty} \sum_{i \in \Z} A_i(x^n) $ and
$\tilde{x}$ is the infimum (or greatest lower bound) of $J$.
%We now show the first inequality of \eqref{eq:lower_conti}.
The step of the proof in Lemma \ref{lemm:lower} implies that for any $n \in \N$:
%and $j \in \N$:
\begin{align}
\label{eq:bdd_low}
%\sum_{|i|>j} A_i(x^n) \ge -C \sum_{|i|>j} \max \{ \rho_{2i}, \rho_{2i+1}\} \ge -C  \sum_{|i|>j}  \rho_i.
\lim_{j \to \infty} \sum_{|i|>j} A_i(x^n) \ge 0
\end{align}
%Since $\sum_{i \in \Z}  \rho_i$ is finite, we have $\sum_{|i|>j}  \rho_i < \epsilon/C$ for sufficiently large $j$.
Hence, the first inequality holds.
The second inequality is clear since $\tilde{x} \in \xkr$ and $\sum_{i \in \Z} \rho_i < \infty$.

\if0
To check the second inequality of \eqref{eq:lower_conti}, it suffices to show that  the value of $J(\tilde{x}) $  is finite.
If $J(\tilde{x}) $ is infinite, then for any $M>0$, there is a $j_0=j_0(M) \in \N$ such that:
\begin{align*}
M \le \sum_{|i| \le j_0} A_i(\tilde{x}) = \sum_{|i| \le j_0} A_i(\lim_{n \to \infty} x^n) = \lim_{n \to \infty} \sum_{|i| \le j_0} A_i(x^n),
\end{align*}
since a finite sum $\sum_{|i| \le j_0} A_i(x)$ is continuous.
On the other hand, for any $\delta \in (0,1)$, there exists $n_0$ such that if $n \ge n_0$,
$\displaystyle{J(x^n)< \inf_{x \in \xkr} J(x) + \delta}$.
Moreover,
for any $\epsilon >0$, there exists $j_1=j_1(\epsilon)$ such that for any $j \ge j_1$,
$| \sum_{|i|>j} A_i(x^n)| < \epsilon$.
Then, for any $n \ge n_0$ and any $\epsilon >0$, we get:
\begin{align*}
M
\le \sum_{|i| \le j_0} A_i(x^n) 
&=  J(x^n) - \sum_{|i| > j_0} A_i(x^n)\\
&< \inf_{x \in \xkr} J(x) + \delta + \epsilon,
\end{align*}
so $\sum_{|i| \le j_0} A_i(x^n) $ is finite, which is a contradiction.
\fi
\end{proof}

%By Lemma \ref{lemm:lower} and \ref{lemm:semiconti}, we obtain

\subsection{Properties of the minimizers of $J$ in $\xkr$}
Let $x^\ast=(x^\ast_i)_{i \in \Z}$ be a minimizer (depending on $k \in K$ and $\rho \in P$) in Proposition \ref{prop:min}.
%First, we show the properties of $A_i$ for $i \equiv 0,2$ when $\rho_i$'s are sufficiently small.
Let $x(n;a,b)=\{x_i(n;a,b)\}_{i=0}^{n}$ be a minimizing sequence of $\sum_{i=0}^{n-1} {h({x_i,x_{i+1})}}$ on $X(n)$
(defined by $\eqref{eq:x}$)
that satisfies $x_0(n;a,b)=a$, and $x_n(n;a,b)=b$, i.e.,
it holds that $\sum_{i=0}^{n-1} {h({x_i(n;a,b),x_{i+1}(n;a,b))}} = \calc(n;a,b)$  (see \eqref{def:kp}).

\begin{lemm}
\label{lemm:distance}
For any $\epsilon \in (0, \min\{c_{\ast}/(2C), (u^1-u^0)/2\})$
($c_\ast$ is given by $\eqref{cstar}$)
, 
there exist two positive real numbers $r_1$ and $r_2$ which satisfy that:
for any $n \ge 2$,  $a \in [u^0,u^0+r_1]$ (resp. $a \in [u^1-r_1,u^1]$), and $b \in [u^0,u^0+r_2]$ (resp. $b \in [u^1-r_2,u^1]$), 
\[
0< x_i(n;a,b)-u^0 <\epsilon \ (resp. 0<u^1 - x_i(n;a,b) <\epsilon) \ \text{for all} \ i \in \{0, \dots,n\}.
\]
\end{lemm}
\begin{proof}
For any $\epsilon \in (0, \min\{c_{\ast}/(2C), (u^1-u^0)/2\})$, we can take $r_1$ and $r_2 \in (0, \epsilon)$ satisfying
\[r_1 +r_2 <  \min \left\{\frac{\phi(\epsilon)}{2C},\frac{c_{\ast}}{2C} -  \epsilon \right\}. \]
See $\eqref{phi}$ for the definition of $\phi$.
We demonstrate that the claim holds for the selected $r_1$ and $r_2$ in the above.
(We only  prove the case of $a \in [u^0,u^0+r_1]$ and $b \in [u^0,u^0+r_2]$. The proof of the other case is similar.)
Set a finite sequence $y=( y_i )_{i=0}^{n}$ by $y_0 = x_0$, $y_n = x_n$, and $y_i=u^0$ otherwise.
For any $n \in \N$,
\begin{align}
\label{val_xnab}
\begin{split}
\sum_{i=0}^{n-1} a_i (x) 
&=\calc(n;a,b) -nc= \sum_{i=0}^{n-1}({h(x_i,x_{i+1})} -h(u^0,u^0))\\
&\le  \sum_{i=0}^{n-1}({h(y_i,y_{i+1})}-h(u^0,u^0))   \le C(r_1+r_2) .
\end{split}
\end{align}
If there exists $i \in \{1,\dots,n-1\}$ satisfying  $x_i-u^0 \ge \epsilon$ and $u^1 - x_i  \ge \epsilon$,
Combining Lemma \ref{lemm:yu_lower} with $\eqref{val_xnab}$ yields:
\[
C(r_1+r_2) \ge \calc(n;a,b) -nc \ge \phi(\epsilon) - C|a-b| >  \phi(\epsilon) - C(r_1+r_2).
\]
Thus we get   $\phi(\epsilon) < 2C(r_1+r_2)$, which is a contradiction.
Next, we assume that there exist $i$ and $i+1$ such that  $u^1 - x_i  < \epsilon$ and $x_{i+1}-u^{0}<\epsilon$.
%for all $i \in \{1,\dots,n-1\}$.
For simplicity, we can set $i=1$ without loss of generality.
Define a configuration $z^+$ and $z^-$ by:
\[
z^+=
\begin{cases}
u^0 \ (i \le 0)\\
x_i \ (i=1) , \\
u^1 \ (i \ge 2)
\end{cases}
z^-=
\begin{cases}
u^1 \ (i \le 0)\\
x_i \ (1 \le i \le n-1).\\
u^0 \ (i \ge n)
\end{cases}
\]
Applying $c=h(u^0,u^0)=h(u^1,u^1)$, Lipschitz continuity, and \eqref{val_xnab}, we see that:
\begin{align*}
I(z^+) + I(z^-)
&= a_0(z^+) +  \sum_{i=1}^{n-1} a_i(z^-)\\
&< \sum_{i=0}^{n-1} a_i (x) + C(r_1+r_2 + 2\epsilon)\\
&\le C(r_1+r_2) + C(r_1+r_2 + 2\epsilon).
\end{align*}
On the other hand,  $z^+ \in X^0$ and $z^- \in X^1$ imply $I(z^+) + I(z^-) \ge c_{\ast}$
and we get:
\[
c_{\ast} < 2C(r_1+r_2 + \epsilon),
\]
which is a contradiction.
\end{proof}

\begin{lemm}
\label{lemm:distance_2}
Assume that both $a - u^0$ and $b-u^0$ (resp. both $u^1-a$ and $u^1-b$) are small enough.
Then, for any $\delta>0$ and $m \in \N$, there exists $N \in \N$ such that for any $n \ge N$, there exist $i(1), \dots, i(m) \in \{0, \cdots, n\}$  satisfying:
\begin{align}
\label{delta}
 x_{i(j)}(n;a,b) -u^0 < \delta \ (resp. \ u^1 -x_{i(j)}(n;a,b) < \delta) \ \text{for all $j \in \{1, \cdots,m\}$}
\end{align}
\end{lemm}
\begin{proof}
If we replace $\eqref{delta}$ with the following:
\[
 x_{i(j)}(n;a,b) -u^0 < \delta \ \text{or} \ u^1 -x_{i(j)}(n;a,b) < \delta) \ \text{for all $j \in \{1, \cdots,m\}$},
\]
our statement for $m=1$ is immediately shown from Proposition \ref{prop:finite} and $\eqref{val_xnab}$.
For $m \ge 2$, 
since both $a - u^0$ and $b-u^0$ are small enough, Lemma \ref{lemm:distance} is valid and it implies \eqref{delta}.
\end{proof}

Those statements of the above two lemmas may seem a bit complicated.
We will roughly summarize the statement of Lemma \ref{lemm:distance} and \ref{lemm:distance_2}.
The former states that
for any $\epsilon$,
if two endpoints are close to $u^0$ or $u^1$,
then a minimal configuration between them is in a band whose width is $\epsilon$ independent of its length.
On the other hand, Lemma \ref{lemm:distance_2} shows that no matter the width of the band, if we take the interval between the endpoints to be longer
(i.e. if we make the length of the band sufficiently long), a minimal configuration can get arbitrarily close to either $u^0$ or $u^1$.

Now we are ready to state our main theorem when the rotation number $\alpha$ is zero:

\begin{proof}[Proof of Theorem \ref{theo:main} for $\alpha=0$]
To see that a minimizer $x^\ast \in \xkr$ is a stationary configuration,
it suffices to show that $x^\ast$ is not on the boundary of $\xkr$.
For any positive sequence $\epsilon = ( \epsilon_i)_{i \in \Z}$ with $\epsilon_i < c_{\ast}/4C$,
we choose $\rho \in P$ and $k \in K$ in the following steps:
\begin{itemize}
\item[Step $1$] 
Since we assume that $ (u^0,u^1) \neq  I^+_{0}(u^0,u^1)$ and $(u^0,u^1) \neq I^-_{0}(u^0,u^1)$,
both $(u^0,u^1) \backslash I^+_{0}(u^0,u^1)$ and $(u^0,u^1) \backslash I^-_{0}(u^0,u^1)$ are nonempty and
we can take $\rho \in P$
so that:
\begin{itemize}
\item[$(p_1)$] $u^{i+1} + \sigma(i) \rho_i \in (u^0,u^1) \backslash I^+_{0}(u^0,u^1)$ for all $i \equiv 1,2$ and $u^i - \sigma(i)  \rho_i \in  (u^0,u^1) \backslash I^-_{0}(u^0,u^1)$ for all $i \equiv -1,0$, where $\sigma(i)=1$ if $i$ is odd and $\sigma(i)=-1$ if $i$ is even.
and
\item[$(p_2)$] For any $i \in \Z$, 
$\displaystyle{
\rho_{2i} + \rho_{2i+1} <  \min \left\{\frac{\phi(\epsilon_i)}{2C},\frac{c_{\ast}}{2C} -  \epsilon_i \right\}
}$,
\end{itemize}
where $u^i=u^0$ when $i$ is even, and $u^i=u^1$ when $i$ is odd.
It easily follows from Lemma \ref{lemm:distance} that, for any $k \in K$ and $\rho \in P$ satisfying $(p_2)$,
a minimizer $x^\ast=(x^\ast_j)_{j \in \Z} \in \xkr$ is on an $\epsilon_i$-neighborhood of $u^i$ for each $j \in [k_{2i}, k_{2i+1}]$,
i.e., $|x^\ast_j - u^i|< \epsilon_i $ if  $j \in [k_{2i}, k_{2i+1}]$.
Notice that 
$\rho_i$ can be chosen as arbitrarily small
since a monotone heteroclinic configuration (one transition orbit),
say $(x_i)_{i \in \Z}$, satisfies $|x_i - u^j| \to 0$ as $|i| \to \infty$ for $j=0$ or $1$
and it is follows from $(h_1)$ that if $x=(x_i)_{i \in \Z}$ is a stationary configuration whose rotation number is $0$,
then so is $y=(y_i)_{i \in \Z}$ with $y_i = x_{i+l}$ for any $l \in \Z$.

\item[Step $2$]
Next, we consider taking a $k \in K$ dependent of the chosen $\rho \in P$ in the previous step.
Since  $k \in K$ is $k_0=0$,
to take $k$ is to determine the values of
$|k_{2i-1} - k_{2i} |$ and
 $|k_{2i} - k_{2i+1} |$ for all $i \in \Z$.
 For each $i \in \Z$, we take the value of $|k_{2i-1} - k_{2i} |$ satisfying:
 \[
 \calm^{i}(u^0,u^1) \cap Y^{i+1}(k_{2(i+1)}, \rho_{2(i+1)}) \cap Y^{i}(k_{2i+1}, \rho_{2i+1}) \neq \emptyset,
\]
i.e.,
\[
 \calm^{i}(u^0,u^1) \cap Y^{i}(0, \rho_{2i+1}) \cap Y^{i+1}(|k_{2i+1}-k_{2(i+1)}|, \rho_{2i}) \neq \emptyset,
\]
where $\calm^i=\calm^0$ and $Y^i=Y^0$ when $i$ is even, and $\calm^i=\calm^1$ and $Y^i=Y^1$ when $i$ is odd.

\item[Step $3$]
Before describing how we choose the value of $|k_{2i} - k_{2i+1} |$ for each $i \in \Z$,
we define several positive bi-infinite sequences.
Set $\tilde{e}=(\tilde{e}_i)_{i \in \Z}$  by:
\begin{align*}
\tilde{e}_i=
\begin{cases}
e_0(\rho_{2i+1},\rho_{2(i+1)}), & (i : \text{even})\\
e_1(\rho_{2i+1},\rho_{2(i+1)}) & (i : \text{odd}).
\end{cases}
\end{align*}
Furthermore,
let's choose two positive sequence  $\delta=(\delta_i)_{i \in \Z}$ and $\tilde{\epsilon}=(\tilde{\epsilon}_i)_{i \in \Z}$
satisfying, for each $i \in \Z$, 
\[2 \tilde\epsilon_i +C( \delta_{2i -1}  + \delta_{2i})< \frac{\tilde{e}^i}{2}.\]

\item[Step $4$]
For each $\tilde\epsilon_i$ in the one before step, 
Lemma \ref{lemm:finite_het} and \ref{lemm:finite_fake_het} show that
there exist two integers $N_{2i-1}, N_{2i} \in \N$ and $x^{i} \in  \calm^i(u^0,u^1)$
satisfying the following:
\begin{itemize}
\item For any $i \in \Z$, if $n_{2i+1} \ge N_{2i+1}$ and $n_{2(i+1)} \ge N_{2(i+1)}$, then both the folllowing 1 and 2 hold:
\begin{enumerate}
\item $\displaystyle{\sum_{j=k_{2i+1}-n_{2i+1}-1}^{k_{2(i+1)} +n_{2(i+1)}+1} a_j(x^i) \le c_i +\tilde{ \epsilon}_i}$
\item $\displaystyle{\sum_{j=k_{2i+1}-n_{2i+1}-1}^{k_{2(i+1)} +n_{2(i+1)}+1} a_j(y) \ge c_i + \tilde{e}_i - \tilde{\epsilon}_i}$
for all:
%$\displaystyle
\begin{align*}
y \in
\{
x=(x_i)_{i \in \Z} \in
X^i \cap  Y^{i}(k_{2i+1}, \rho_{2i+1}) &\cap Y^{i+1}(k_{2(i+1)}, \rho_{2(i+1)})  \\
&
\mid
|x_{k_{2i+1}} - u^{i}|= \rho_{2i+1} \ \text{or} \ |x_{k_{2(i+1)}} - u^{i+1}|= \rho_{2(i+1)}
\},
\end{align*}
\end{enumerate}
\end{itemize}
where $X^i=X^0$ when $i$ is even, and $X^i=X^1$ when $i$ is odd.
For the above ${(N_i)}_{i \in \Z}$, we construct in Step $4$, we take $|k_{2i}-k_{2i+1}|$ so that
\begin{itemize}
\item[$(k_1)$] $|k_{2i}-k_{2i+1}| \gg N_{2i} + N_{2i+1}$
\end{itemize} 
We will give more precise conditions for $k$ and explain the role of the sequence $\delta$ later.
\end{itemize}

We shall show that our statement holds for $\rho$ and $k$ along with the above steps.
Combining Lemmas \ref{lemm:estimate_periodic1} and \ref{lemm:estimate_periodic2} implies that $x^\ast_i \notin \{u^0,u^1\}$ for all $i \in \Z$.
We assume $x^\ast_{k_1} = u^0 + \rho_1$ for a contradiction argument.
Let $y=(y_i)_{i \in \Z}$ be:
\begin{align*}
y_i=
\begin{cases}
     x_i^\ast & i < k_1-n_1 \ \text{or} \ i >k_2+n_2\\
     x^1_i &  i \in [k_1-n_1,k_2+n_2 ]
\end{cases},
\end{align*}
where $x^1 = (x_i^1)_{i\in\Z}$ is given in Step $4$.
%Hereafter, $x_i^\ast$ will be written simply as $x_i$ unless there is potential for confusion.
It is sufficient to show that:
\[
\sum_{i \in [k_1-n_1-1,k_2+n_2+1 ]} A_i(y) < \sum_{i \in [k_1-l_1-1,k_2+n_2+1 ]} A_i(x^\ast).
\]
Applying Lemma \ref{lemm:distance} and \ref{lemm:distance_2} for $m=2$ shows that,
for $|k_{2i}-k_{2i+1}|$ sufficiently large, there exist $n_{2i-1} \ge N_{2i-1}$ and $n_{2i} \ge N_{2i}$ such that:
\[
| x^{i-1}_{k_{2i-1} + \sigma(i) n_{2i-1}}- y_{k_{2i-1} + \sigma(i) n_{2i-1}} | \le \delta_{2i-1}, \text{and} \  | x^i_{k_{2i} - \sigma(i)  n_{2i}} - y_{k_{2i}  - \sigma(i) n_{2i}} | \le \delta_{2i}.
\]
Applying Step $4$ and Lipschitz continuity of $h$, we see that:
\begin{align*}
& \sum_{i \in [k_1-n_1-1,k_2+n_2+1 ]}( A_i(x^\ast) - A_{i}(y)) = \sum_{i \in [k_1-n_1-1,k_2+n_2+1 ]}( a_i(x^\ast) - a_{i}(y))\\
 &\ge \tilde{e}_1 - 2\tilde{\epsilon}_1  + (h(x_{k_1-n_1-1},x_{k_1-n_1}) - h(x_{k_1-n_1-1},y_{k_1-n_1}) )+  (h(x_{k_2+n_2},x_{k_2+n_2+1}) - h(y_{k_2+n_2},x_{k_2+n_2+1})) \\
 &\ge  \tilde{e}_1 - 2\tilde{\epsilon}_1  - C(\delta_{1} + \delta_{2}) > \frac{\tilde{e}_1}{2}> 0.
\end{align*}
This completes the proof.
\end{proof}
%%%%%%
%以下，査読業者のチェックなし

\if0
\begin{theo}
For any $\epsilon=(\epsilon_{2i})_{i \in \Z}$ with $\epsilon_i >0$,
there are sequences $\rho=(\rho_i)_{i \in \Z} \in P$ and $\tilde{I}=(\tilde{I}_i)_{i \in \Z}$ such that
for any $k =(k_i)_{i \in \Z} \in K$ with $|k_i-k_{i+1}| \ge \tilde{I}_i$ for all $i \in \Z$,
for each $i \in \Z$, $|x_j^\ast-u^i| \le \epsilon_{j}$ for all $j \in [k_i,k_{i+1}] \cap \Z$
 where $x^\ast_{k, \rho}=(x^\ast_i)_{i \in \Z}$ 
\end{theo}
\begin{proof}
\end{proof}
\fi

Moreover, we immediately obtain the following:
\begin{proof}[Proof of Theorem \ref{theo:main} for $\alpha \in \Q \backslash \{0\}$]
Let
$x^+=(x^{+}_i)_{i \in \Z}$ and $x^-=(x^{-}_i)_{i \in \Z} \in \calm_{\alpha}^{\per}$ satisfy
that $(x^{-}_0, x^{+}_0)$ is a neighboring pair of $\calm_{\alpha}^{\per}$,
and
$(x_0,x_1, \cdots,x_q)$ satisfy $h(x_{0}, \cdots, x_{q}) =H(x_{0},x_{q})$.
It is clear that
if $|x_0-x_0^{\pm}| \to 0$ and $|x_q-x_q^{\pm}| \to 0$,
then  $|x_i-x_i^{\pm} | \to 0$ for all $i=1, \cdots, q-1$,
so the obtained local minimizer with rotation number $\alpha$ holds the second property of Theorem \ref{theo:main}.
Thus we get a desired stationary configuration from Proposition \ref{prop:alpha_configuration} and the proofs of Theorem \ref{theo:main} for $\alpha=0$.
\end{proof}

Though the previous discussion treats bi-infinite transition orbits,
we can construct one-sided infinite transition orbits
by replacing $c(j)$ of $J$ with:
\begin{align*}
\tilde{c}(j) =
\begin{dcases}
\frac{\calc(|I_{2i+1}|,u^i,u^{i+1})}{|I_{2i+1}|}&\ (j \in I_{2i+1} \ \text{for some} \  i \ge 0),\\
\frac{\calc(|I_{2i}|,u^i,u^i)}{|I_{2i}|}  & \ (\text{otherwise}),
\end{dcases}
\end{align*}
%(see $\eqref{c}$ for the definition of $c$)
and
$\xkr$ with:
\[
\tilde{X}_{k,\rho}(a,b)
=
 \bigcap_{i \in \Z}
 \left\{
\left( \bigcap_{i < 0}  Y^a(k_{bi},\rho_{0}) \right) \cap
\left( \bigcap_{i \equiv 0,1, i \ge 0}  Y^a(k_{bi},\rho_{bi}) \right) \cap \left( \bigcap_{i \equiv -1,2, i \ge 0}  Y^{|1-a|}(k_{bi},\rho_{bi}) \right)
\right\}
,
\]
where $a \in \{0,1\}$, $b \in \{-1,1\}$, $k \in K$ and $\rho \in P$.
Let $\tilde{J}$ be the replaced function instead of $J$, i.e.,
\[
\tilde{J}(x) = \sum_{j \in \Z} (h(x_j,x_{j+1})-\tilde{c}(j)).
\]
Notice that Proposition \ref{prop:finite} implies that if $x=(x_i)_{i \in \Z}$ satisfies that $\tilde{J}(x)$ is finite,
then $|x_i - x^{a}_{i}|\to 0$ \ $(bi \to \infty)$.
\if0
\begin{figure}[htbp]
    \centering
    \includegraphics[width=8.5cm]{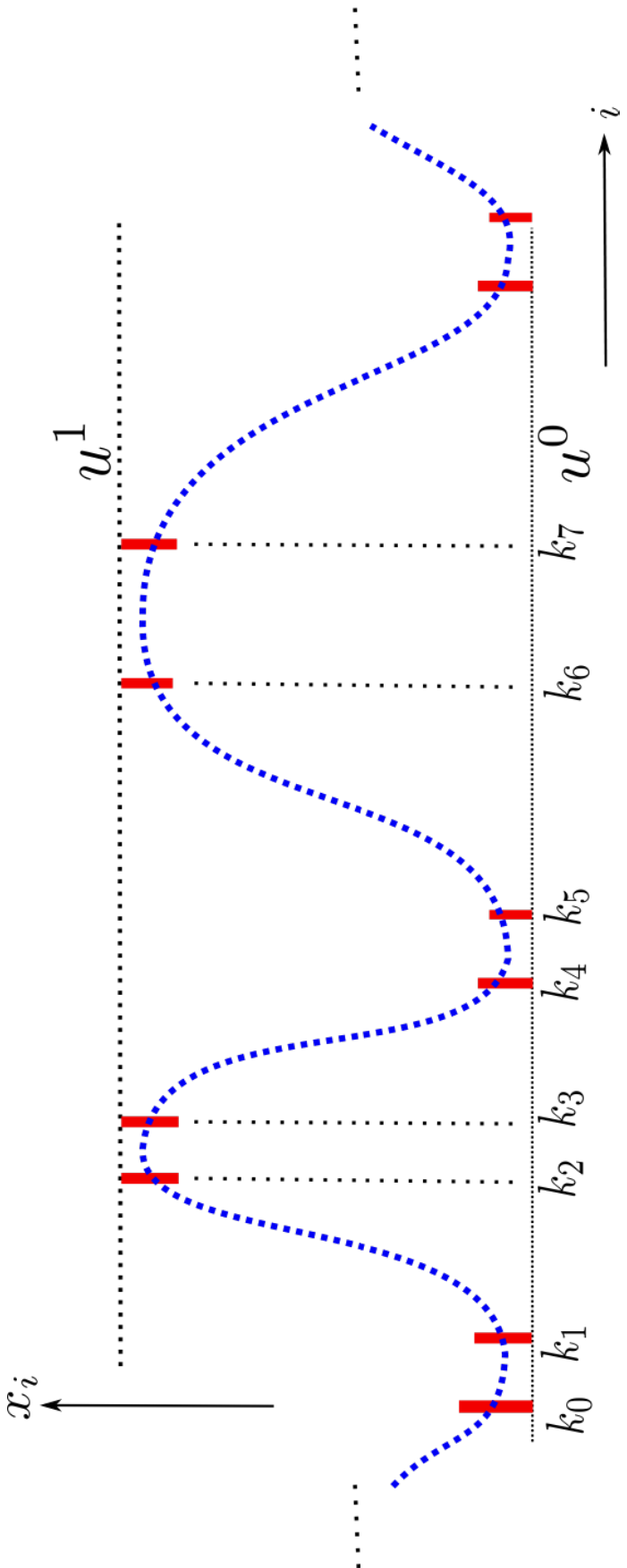}
  \caption{An element of $\tilde{X}_{k,\rho}(0,1)$}
\end{figure}
\fi
Thus we get:
\begin{theo}[]
\label{theo:main_2}
Assume the same condition of Theorem \ref{theo:yu_main}.
Then, for
any $a \in \{0,1\}$, $b \in \{-1,1\}$, and
positive sequence ${\epsilon}=(\epsilon_i)_{i \in \Z}$ with $\epsilon_i$ small enough,
there is an $m=\{m_i\}_{i \in \Z}$
such that for every sequence of integers $k=(k_i)_{i \in \Z}$ with $k_{i+1}-k_{i} \ge m_i$,
there is a stationary configuration $x$ satisfying:
\begin{enumerate}
\item $x_i^0<x_i<x_{i}^1$ for all $i \in \Z$;
\item for any $j \in \Z$, \ $|x_{i}-x^{2j+a}_{i}| \le \epsilon_i$  if $i \in [k_{4j}, k_{4j+1}]$ \text{and} \, 
$|x_{i}-x^{2j-1+a}_{i}| \le \epsilon_i$ if $i \in [k_{4j+2}, k_{4j+3}]$;
\item $|x_i - x^{a}_{i}|\to 0$ \ $(bi \to \infty)$.
\end{enumerate}
%For any $j \in \Z$, $x^j=x^0$, if $j$ is even, and $x^j=x^1$, if $j$ is odd.
%Under the same assumption as Theorem \ref{theo:yu_main},
%there exist uncountably infinitely many $\infty$-transition orbits.
\end{theo}

%Billiard maps
\if0
\subsection{Billiard maps}
Next, we consider a billiard map that does not satisfy the twist condition.
Let $f \colon \R \to \R$ be a  positive smooth function satisfying $f(x)=f(x+1)$ for all $x \in \R$.
Set an area $D=D(f)$ by
\[
D = \{ (x,y) \in \R^2 \mid -f(x) \le y \le f(x)\}.
\]
Let $s$ be an arc-length parameter, i.e., $s$ is given by
\[
s=\int_{0}^{x} \sqrt{1 + f'(\tau)^2} d \tau =:g(x)
\]
and $x$ is represented by $x=g^{-1}(s)$. Moreover, we set
\[
\tilde{f}(s) = f(g^{-1}(s)).
\]
We can see a variational structure of billiard maps for the above settings.
Consider
\[
h(s,s') = d_{\mathrm{E}}(a_{+}(s),a_{-}(s')),
\]
where 
$a_{\pm}(s):=(g^-(s), \pm \tilde{f}(s))$
and
$d_{\mathrm{E}}$ is Euclidean metric on $\R^2$.
We check that $h$ satisfies $(h_1)$-$(h_5)$.
\fi

\section{Additional remarks}
\label{sec:add}
\subsection{The number of infinite transition orbits}
\label{sec:general}
We first see that Theorem \ref{theo:main} and \ref{theo:main_2} show the existence of uncountable many infinite transition orbits.
We can take $k \in K$ and $\rho \in P$ given in Theorem \ref{theo:main} so that for all $i \in \N$, $k_{i}-k_{i-1} < k_{i+1}-k_{i} $ and $ k_{-i} - k_{-(i+1)} <k_{-i+1}-k_{-i} $.
For $j=(j_i)_{i \in \Z} \in K$, set:
\begin{align*}
X_j=
 \bigcap_{i \in \Z}
 \left\{
\left( \bigcap_{i \equiv 0,1}  Y^0(k_{j_i},\rho_i) \right) \cap \left( \bigcap_{i \equiv -1,2}  Y^1(k_{j_i},\rho_i) \right)
\right\}
\end{align*}
Let  $x^\ast(j)$ be a minimizer of $J$ on $X_j$, i.e., 
\[
J(x^\ast(j))=\inf_{x \in X_j} J(x).
\]
%Let $j^0=(j^0_i)_{i \in \Z}$ be a sequence of $j^0_i=i$.
The previous section deals with the case of $j^0=(j^0_i=i)_{i \in \Z}$.
It is easily seen that 
if $l \neq m \in K$, then $x^\ast(l)$ and $x^\ast(m)$ are different and
we immediately get the following theorem.
\begin{theo}
Let $\# \chi_1 $ and $\# \chi_2$   be the number of infinite transition orbits in Theorem  \ref{theo:main} and \ref{theo:main_2}.
Then $\# \chi_1 =\# \chi_2 = \# \R$.
\end{theo}
\begin{proof}
We only discuss the case of Theorem  \ref{theo:main}.
For any real number $r \in \R_{>0}$, we can choose a corresponding bi-infinite sequence ${(a_i)}_{i \in \Z} \subset \Z_{\ge 0}$.
(For example,  when $r=12.34$, $a_{-1}=1, a_0=2, a_1=3, a_2=4$ and $a_i =0$ otherwise.)
The proof is straightforward by setting $a_i := j_{i+1} - j_i -1$ for $i \in \Z$.
It is also clear that if $r_1 \neq r_2$, 
each corresponding stationary configuration is different.
A similar proof is valid for Theorem \ref{theo:main_2}.
\end{proof}

\subsection{A special case}
\label{sec:example_h}
We will give a special example at the end of this paper.
In the previous section,  we cannot generally show:
\begin{align}
\label{per_min_ineq}
h(x,y) -c \ge 0.
\end{align}
Therefore the proof of Proposition \ref{prop:finite} is somewhat technical.
However, as we will see later, \eqref{per_min_ineq} holds if $h$ satisfies:
\begin{align}
\label{rev}
h(x,y) = h(y,x).
\end{align}
This is kind of natural because the analogy of \eqref{per_min_ineq} for differential equations
 holds in variational structures of potential systems with reversibility (see \cite{Rabinowitz2008}).
One of the examples satisfying  \eqref{rev} is the Frenkel-Kontorova model  \cite{AubrDaer1983, FrenKont1939}
and the corresponding $h$ is given by:
\begin{align}
\label{fkmodel}
h(x,y)=\frac{1}{2} \left\{ C(x-y)^2 + V(x) + V(y) \right\},
\end{align}
where $C$ is a positive constant and $V(x)=V(x+1)$ for all $x \in \R $.
Since  $\partial_1 \partial_2 h \le -C < 0$,
Remark \ref{rem:delta_bangert} implies that \eqref{fkmodel} satisfies $(h_1)$-$(h_5)$.
Using $\eqref{rev}$, we can easily show the following lemma, which implies $h(x,y) -c \ge 0$.
\begin{lemm}
If a continuous function $h \colon \R^2 \to \R$ satisfies ($h_1$)-($h_3$) and $\eqref{rev}$,
then all minimizers of $h$ are $(1,0)$-periodic, i.e.,
\[
\inf_{x \in \R} h(x,x) = \inf_{(x,y) \in \R^2} h(x,y).
\]
\end{lemm}
\begin{proof}
First, we see that it follows from $(h_2)$ that there exists an infimum of $ h(x,y)$ on $\R^2$.
From $(h_1)$, we can choice $x^\ast$ satisfying $h(x^\ast,x^\ast)=\min_{x \in [0,1]} h(x,x)=\inf_{x \in \R} h(x,x) $.
By contradiction, there is $(x,y) \in \R^2$ such that $x \neq y$ and $h(x,y) < h(x^\ast,x^\ast)$.
Then, $\eqref{rev}$ implies:
\[
h(x,y) + h(y,x) < h(x^\ast,x^\ast) + h(x^\ast,x^\ast) \le  h(x,x) + h(y,y),
\]
but it contradicts $(h_3)$.
\end{proof}

If $h$ satisfies \eqref{rev},
minimal configurations are `almost' monotone in the following sense:
\begin{prop}
Let $n \in \N$ be arbitrary number
and $x=(x_{i})_{i=0}^{n}$ be a finite configuration with $x_0=a$, $x_n=b$ and $a<b$ \ (resp. $a>b$).
If there exist two integers $m$ and $l$
satisfying $0<m<n$,  $0\le m-l<m+l+1 \le n$,
 $x_{m}> x_{m+1}$ (resp. $x_{m}< x_{m+1}$) and $x_{m-l}<x_{m+l+1}$ \ (resp. $ x_{m-l}>x_{m+l+1}$),
 then $x=(x_{i})_{i=0}^{n}$ is not minimal.
\end{prop}
\begin{proof}
%monotone
We only consider the case where  $l=1$.
To prove our statement, it suffices to construct a finite configuration $y=(y_i)_{i=m-1}^{m+2}$
satisfying
$\sum_{i=m-l}^{m+l} h(y_i,y_{i+1})< \sum_{i=m-l}^{m+l} h(x_i,x_{i+1})$.
Set $y=(y_i)_{i=m-1}^{m+2}$
by $y_i = x_{i}$ \ $(i=m-1,m+2)$, $y_m=x_{m+1}$, and $y_{m+1} = x_{m}$.
Applying $\eqref{rev}$ and $(h_3)$, we have:
\begin{align*}
 &\sum_{i=m-1}^{m+1} h(x_i,x_{i+1}) -  \sum_{i=m-1}^{m+1} h(y_i,y_{i+1})\\
& =h(x_{m-1},x_m) + h(x_{m+1},x_{m+2}) -h(x_{m-1},x_{m+1}) + h(x_{m},x_{m+2})\\
& =h(x_{m-1},x_m) + h(x_{m+2},x_{m+1}) -h(x_{m-1},x_{m+1}) + h(x_{m+2},x_{m})>0
\end{align*}
The same reasoning applies to the remaining cases.
This completes the proof.
\end{proof}

\bibliographystyle{amsplain}
\bibliography{kajihara_reference}
%% \bibliography{...}

\end{document}